\documentclass[twocolumn,twoside,10pt]{article}
\usepackage{graphicx,amsmath,amssymb,amsthm}

\newtheorem{theorem} {Theorem}[section]

\newtheorem{definition}{Definition}[section]
\newtheorem{lemma}[theorem]{Lemma}
\newtheorem{cor} [theorem]{Corollary}
\newtheorem{rem}{Remark}[section]

\newcommand{\R}{\mathbb{R}}

\begin{document}
	
	\date{}
	
\title{\Large\bf \uppercase {S-asymptotically $\omega$-periodic solution for a nonlinear differential equation with piecewise constant argument via S-asymptotically $\omega$-periodic functions in the Stepanov sense} }
\author{\Large\bf Author Index}
\author{William Dimbour  \thanks{ William Dimbour, UMR Espace-Dev Universit\'e de Guyane, 
		Campus de Troubiran 97300 Cayenne  Guyane E-mail: William.Dimbour@espe-guyane.fr}, Solym Mawaki Manou-Abi \thanks{  Solym Mawaki MANOU-ABI, Institut Montpellierain Alexander Grothendieck, Centre Universitaire de Mayotte, 3 Route Nationale 97660 Dembeni  E-mail: solym-mawaki.manou-abi@univ-mayotte.fr}}
\maketitle


{\footnotesize \noindent	{\bf Abstract.}  	In this paper, we show the existence of function which is not $S$-asymptotically $\omega$-periodic, but  which is $S$-asymptotically $\omega$-periodic in the Stepanov sense. We give sufficient conditions for the existence and uniqueness of $S$-asymptotically $\omega$-periodic solutions for a nonautonomous differential equation with piecewise constant argument in a Banach space when $\omega$ is 	an integer. This is done using the Banach fixed point Theorem. An example involving the heat operator is discussed as an illustration of the theory.\\

}



\section{Introduction}
In this paper, we study the existence and uniqueness of $S$-asymptotically $\omega$-periodic solution of the following differential equation with piecewise constant argument	
\begin{equation}
\label{eqn: eq3050}
\left\{
\begin{array}{l}
x'(t)=A(t)x(t)+f(t,x([t]) ), \\
x(0)= c_{0},
\end{array}
\right.
\end{equation}
where $\mathbb{X}$ is a banach space, $c_0 \in \mathbb{X}$, $[\cdot]$ is the largest integer function, f is a continuous function on $\mathbb{R}^+ \times \mathbb{X}$ and $A(t)$ generates an exponentially stable evolutionnary process in $\mathbb{X}$.
The study of differential equations with piecewise constant argument (EPCA) is an important subject because these equations have the structure of continuous dynamical systems in intervals of unit length. Therefore they combine the properties of both differential and difference equations. There have been many papers studying EPCA, see for instance \cite{Wiener}, \cite{Wiener1}, \cite{Wiener2}, \cite{Wiener3}, \cite{Wiener4} and the references therein.

Recently, the concept of $S$-asymptotically $\omega$-periodic function has been introduced in the litterature by Henr\'iquez, Pierri and T\'aboas in \cite{henriquez}, \cite{henriquez1}. In \cite{blot}, the authors studied properties of $S$-asymptotically $\omega$-periodic function taking values in Banach spaces including a theorem of composition.  They applied the results obtained in order to study the existence and uniqueness of $S$-asymptotically $\omega$-periodic mild solution to a nonautonomous semilinear differential equation. In \cite{liang}, the authors established some sufficient conditions about the existence and uniquenes of $S$-asymptotically $\omega$-periodic solutions to a fractionnal integro-differential equation by applying fixed point theorem combined with sectorial operator, where the nonlinear pertubation term $f$ is a Lipschitz and non-Lipschitz case. In \cite{caicedo}, the authors prove the existence and uniqueness of mild solution to some functional differential equations with infinite delay in Banach spaces which approach almost automorphic function (\cite{diagana}, \cite{nguerekata}) at infinity and discuss also the existence of $S$-asymptotically $\omega$-periodic mild solutions. In \cite{xia}, the author discussed about the existence of $S$-asymptotically $\omega$-periodic mild solution of semilinear fractionnal integro-differential equations in Banach space, where the nonlinear pertubation is $S$-asymptotically $\omega$-periodic or $S$-asymptotically $\omega$-periodic in the Stepanov sense (\cite{henrist}, \cite{xia}, \cite{xie}). The reader may also consult \cite{cuevas2009}, \cite{cuevas2010}, \cite{cuevas2013}, \cite{dimbour}, \cite{pierri} in order to obtain more knowledge about $S$-asymptotically $\omega$-periodic functions. Motivated by \cite{blot} and \cite{dimbour}, we will show the existence and uniqueness of $S$-asymptotically $\omega$-periodic solution for (\ref{eqn: eq3050}) where the nonlinear pertubation term $f$ is a $S$-asymptotically $\omega$-periodic function in the Stepanov sense.
The work has four sections. In the next section, we recall some properties about $S$-asymptotically $\omega$-periodic functions. We study also qualitative properties of $S$-asymptotically $\omega$-periodic functions in the Stepanov sense. In particular, we will show the existence of functions which are not $S$-asymptotically $\omega$-periodic but which are $S$-asymptotically $\omega$-periodic in the Stepanov sense. In section 3, we study the existence and uniquenes of $S$-asymptotically $\omega$-periodic mild solutions for (\ref{eqn: eq3050}) considering $S$-asymptotically $\omega$-periodic functions in the Stepanov sense. In section 4, we deal with the existence and uniqueness of $S$-asymptotically $\omega$-periodic solution for a partial differential equation.
\section{Preliminaries}
\begin{definition}(\cite{henriquez})
	A function $f \in BC(\mathbb{R}^{+}, \mathbb{X})$ is called $S$-asymptotically $\omega$ periodic if there exists $\omega$ such that $\displaystyle{\lim_{t \rightarrow \infty}}(f(t+\omega)-f(t))=0$. In this case we say that $\omega$ is an asymptotic period of $f$ and that $f$ is $S$-asymptotically $\omega$ periodic. The set of all such functions will be denoted by $SAP_{\omega}(\R^+, \mathbb{X})$.
\end{definition}

\begin{definition}(\cite{henriquez})
	A continuous function $f :\mathbb{R}^{+}\times \mathbb{X} \rightarrow \mathbb{X}$ is said to be uniformly $S$-asymptotically $\omega$ periodic on bounded sets  if for every bounded set $K^{*}\subset \mathbb{X}$, the set $\{f(t,x):t \ge 0,x\in K^{*}\}$ is bounded and $$\displaystyle{\lim_{t \rightarrow \infty}}
	(f(t,x)-f(t+\omega,x))=0$$
	uniformly in $x\in K^{*}$.
\end{definition}

\begin{definition}(\cite{henriquez})
	A continuous function $f :\mathbb{R}^{+}\times \mathbb{X} \rightarrow \mathbb{X}$ is said to be asymptotically uniformly continuous on bounded sets if for every  $\epsilon>0$ and every bounded set $K^{*}$, there exist $L_{\epsilon, K^{*}}>0$ and $\delta_{\epsilon, K^{*}}>0$ such that $\vert\vert f(t,x)-f(t,y)\vert\vert < \epsilon$ for all $t\ge L_{\epsilon,K^{*}}$ and all $x,y \in K^{*}$ with $\vert\vert x-y \vert\vert< \delta_{\epsilon, K^{*}}$.
\end{definition}

\begin{lemma}(\cite{blot})
	Let $\mathbb{X}$ and $\mathbb{Y}$ be two Banach spaces, and denote by $B(\mathbb{X},\mathbb{Y})$, the space of all bounded linear operators from $\mathbb{X}$ into $\mathbb{Y}$. Let $A \in  B(\mathbb{X},\mathbb{Y})$. Then when $f\in SAP_{\omega}(\R^+, \mathbb{X})$, we have $Af:=[t \rightarrow Af(t)]\in SAP_{\omega}(\R^+, \mathbb{Y})$.
\end{lemma}

\begin{lemma}(\cite{henriquez})
	Let $f :\mathbb{R}^{+}\times \mathbb{X} \rightarrow \mathbb{X}$ be a function which is uniformly $S$-asymptotically $\omega$ periodic on bounded sets and asymptotically uniformly continuous on bounded sets. Let $u: \mathbb{R}^{+} \rightarrow \mathbb{X}$ be $S$-asymptotically $\omega$ periodic function. Then the Nemytskii operator $\phi(\cdot):= f(\cdot,u(\cdot))$ is a $S$-asymptotically $\omega$ periodic function.
\end{lemma}

\begin{lemma}(\cite{liang})
	Assume $f :\mathbb{R}^{+}\times \mathbb{X} \rightarrow \mathbb{X}$ be a function which is uniformly $S$-asymptotically $\omega$ periodic on bounded sets and satisfies the Lipschitz condtion, that is, there exists a constant $L>0$ such that
	$$\vert\vert f(t,x)-f(t,y) \vert\vert \le L\vert\vert x- y \vert\vert, \forall t \ge 0, \forall x,y \in \mathbb{X}.$$
	If $u\in SAP_{\omega}(\mathbb{R}^{+}, \mathbb{X})$, then the function $t \rightarrow f(t,u(t))$ belongs to $SAP_{\omega}(\mathbb{R}^{+}, \mathbb{X})$.
\end{lemma}

Let $p \in[0,\infty[$. The space $BS^{p}(\mathbb{R}^+, \mathbb{X})$ of all Stepanov bounded functions, with the exponent $p$, consists of all measurable functions 
$f : \mathbb{R}^+ \to  \mathbb{X}$  such that $f^{b} \in \mathbb{L}^{\infty}(\mathbb{R},L^{p}([0,1]; \mathbb{X}))$, where
$f^b$ is the Bochner transform of $f$ defined by $f^b(t,s) :=f(t + s), t \in \mathbb{R}^+, s\in [0,1].$
$BS^{p}(\mathbb{R}^+,X)$ is a Banach space with the norm 
$$ ||f||_{S^{p}}= ||f^{b}||_{\mathbb{L}^{\infty}(\mathbb{R}^+,L^{p})}= \sup_{t\in \mathbb{R}^+} \Big(
\int_{t}^{t+1} ||f(\tau)||^{p} d\tau \Big)^{\frac{1}{p}}.               $$
It is obvious that  $L^{p}(\R^+, \mathbb{X})\subset BS^{p}(\mathbb{R}^+, \mathbb{X}) \subset L^{p}_{loc}(\R^+, \mathbb{X})$ and 
$ BS^{p}(\mathbb{R}^+, \mathbb{X}) \subset BS^{q}(\mathbb{R}^+, \mathbb{X}) $ for $p\geq q \geq 1$.
We denote by  $BS^{p}_{0}(\mathbb{R}^+, \mathbb{X})$ the subspace of $BS^{p}(\mathbb{R}^+, \mathbb{X})$ consisting of functions $f$ such that $\int_{t}^{t+1} ||f(s)||^{p}ds \rightarrow 0$ when $t\to \infty$.\\

Now we give the definition of $S$-asymptotically $\omega$-periodic functions in the Stepanov sense.

\begin{definition} \cite{henrist}
	A function $f \in BS^{p}(\mathbb{R}^{+}, \mathbb{X}) $ is called $S$-asymptotically $\omega$-periodic in the Stepanov sense (or $S^p$-$S$-asymptotically $\omega$-periodic)if
	
	$$ \lim_{t\to \infty} \int_{t}^{t+1} ||f(s+\omega)-f(s)||^{p} = 0.   $$
	Denote by $S^{p}SAP_{\omega}(\mathbb{R}^{+}, \mathbb{X})$ the set of such functions.
\end{definition}

\begin{rem}
	It is easy to see that 
	$SAP_{\omega}(\mathbb{R}^{+}, \mathbb{X}) \subset S^{p}SAP_{\omega}(\mathbb{R}^{+}, \mathbb{X}).$
\end{rem}

\begin{lemma}\label{limite 1}
	Let $u\in SAP_{ \omega}(\mathbb{R}^{+},\mathbb{X})$ where  $\omega \in \mathbb{N}^{*}$ 
	, then the step function $t \rightarrow u([t])$ satisfies
	$$\lim_{t \rightarrow \infty}u( \big[ t+\omega \big])-u( \big[ t\big])=0.$$
\end{lemma}
\begin{rem}
	The proof of the above Lemma is contained in the lines of the proof of the Lemma 2 in \cite{dimbour}.
\end{rem}

\begin{cor}
	Let $u\in SAP_{\omega}(\mathbb{R}^{+},\mathbb{X})$ where  $\omega \in \mathbb{N}^{*}$, then the  function $t \rightarrow u([t])$ is $S$-asymptotically $\omega$-periodic in the Stepanov sense but is not 
	$S$-asymptotically $\omega$-periodic.
\end{cor}
\begin{proof}
	By the above Lemma we have :
	$$\forall \epsilon^{1/p}>0, \, \exists T >0; \; t \geq T \Rightarrow ||u([t+\omega])-u([t])|| \leq        \epsilon^{1/p}. $$
	The function $t\to u[t]$ is a step function therefore  it is measurable on $\R_{+}$. Then 
	for $t\geq [T]+1$, we have 
	\begin{eqnarray*}
		\int_{t}^{t+1} ||u([s+\omega])-u([s])||^{p} & \leq & \int_{t}^{t+1} \epsilon ds \\
		& \leq & \epsilon.  
	\end{eqnarray*}
	Therefore the function $t \rightarrow u([t])$ is $S$-asymptotically $\omega$-periodic in the Stepanov sense. Now since the function  $t \rightarrow u([t])$ is not continuous on $\R_{+}$, it can't be $S$-asymptotically $\omega$-periodic. 
\end{proof}
\begin{definition}\cite{henrist}
	A function $f : \R^{+}\times  \mathbb{X} \rightarrow  \mathbb{X} $ is said to be uniformly $S$-asymptotically $\omega$-periodic on bounded sets in the Stepanov sense if for every bounded set $B\subset  \mathbb{X}$,there exist positive functions $g_{b} \in BS^{p}(\mathbb{R}^{+},\R)$ and $h_{b} \in BS_{0}^{p}(\mathbb{R}^{+},\R)$
	such that $f(t,x) \leq g_{b}(t)$ for all $t\geq 0$, $x\in B$  and 
	$|| f(t+\omega,x)-f(t,x)|| \leq  h_{b}(t) $ for all $t\geq 0$, $x \in B$. 
\end{definition}
Denote by $S^{p}SAP_{\omega}(\mathbb{R}^{+}\times  \mathbb{X}, \mathbb{X})$ the set of such functions.

\begin{definition}\label{all1}
	\cite{henrist}
	A function $f : \R^{+}\times  \mathbb{X} \rightarrow  \mathbb{X} $ is said to be asymptotically uniformly continuous on bounded sets in the Stepanov sense if for every $\epsilon >0$ and every bounded set $B\subset X$, there exists  $t_{\epsilon}\geq 0$ and $\delta_{\epsilon} >0$ such that
	
	$$\int_{t}^{t+1} || f(s,x)-f(s,y)||^{p} ds \leq  \epsilon^{p},$$ 
	for all $t\geq t_{\epsilon}$ and all $x,y\in B$ with $||x-y|| \leq \delta_{\epsilon}.$ 
\end{definition}

\begin{lemma}\label{def2}
	\cite{henrist}
	Assume that $f \in S^{p}SAP_{\omega}(\mathbb{R}^{+}\times  \mathbb{X}, \mathbb{X})$  is an asymptotically uniformly continuous on bounded sets in the Stepanov sense function. Let $u \in SAP_{\omega}(\R^{+}, \mathbb{X})$, then 
	$v(.)= f(.,u(.)) \in S^{p}SAP_{\omega}(\mathbb{R}^{+}\times  \mathbb{X}, \mathbb{X}).$
	
\end{lemma}
\begin{lemma}
	\label{limit1f}
	Let $ \omega\in \mathbb{N}^{*}$. Assume $f :\mathbb{R}^{+}\times \mathbb{X} \rightarrow \mathbb{X}$ be a function which is uniformly $S$-asymptotically $\omega$ periodic on bounded sets and satisfies the Lipschitz condition, that is, there exists a constant $L>0$ such that
	$$\vert\vert f(t,x)-f(t,y) \vert\vert \le L\vert\vert x- y \vert\vert, \forall t \ge 0, \forall x,y \in \mathbb{X}.$$
	If $u\in SAP_{\omega}(\mathbb{R}^{+}, \mathbb{X})$, then
	\begin{itemize}
		\item[(1)] the bounded piecewise continuous function $t \rightarrow f(t,u(\big[ t\big]))$ satisfies
		$$\lim_{t \rightarrow \infty}(f(t+\omega,u(\big[ t+\omega \big]))-f(t,u( \big[ t\big]))=0.$$
		\item[(2)] the function $t \rightarrow f(t,u( \big[ t\big]))$ belongs to  $S^{p}SAP_{\omega}(\mathbb{R}^{+}, \mathbb{X}).$
		\item[(3)] the function $t \rightarrow f(t,u( \big[ t\big]))$ does not belongs to $SAP_{\omega}(\mathbb{R}^{+}, \mathbb{X}).$
	\end{itemize}
\end{lemma}
\begin{proof}
	$(1)$
	Since $\mathcal{R}(u)=\{u(\big[ t \big])\vert t\ge 0 \}$ is a bounded set, then for every $\frac{\epsilon}{2}>0$, there exists a constant $L_{\epsilon}>0$ such that 
	$$\vert\vert f(t+\omega,x)-f(t,x) \vert\vert \le \frac{\epsilon}{2}$$
	
	for every $t > L_{\epsilon}$ and $x \in \mathcal{R}(u)$.\\
	By Lemma \ref{limite 1}, for every $\frac{\epsilon}{2L}>0$, there exist $T_{\epsilon}>0$ such that for    
	all $t > T_{\epsilon}$
	$$\vert\vert u( \big[ t+\omega \big])-u( \big[ t\big]) \vert\vert \le \frac{\epsilon}{2L}.$$
	
	We have
	\begin{eqnarray*}
	    &\;\;&	\vert\vert  f(t+\omega,u( \big[ t+\omega\big])) - f(t,u( \big[ t \big])\vert\vert\\
	    &\leq & \vert\vert f(t+\omega,u( \big[ t+\omega\big]))-f(t,u( \big[ t+\omega\big]))\vert\vert\\
		& + &  \vert\vert f(t,u( \big[ t+\omega\big]))-f(t,u( \big[ t \big])\vert\vert \\
		& \leq & 
		\vert\vert f(t+\omega,u( \big[ t+\omega\big]))-f(t,u( \big[ t+\omega\big]))\vert\vert\\
		& + & L \vert\vert u(\big[ t+\omega \big])-u( \big[ t \big])\vert\vert. 
	\end{eqnarray*}
	We put $T=max(T_{\epsilon}, L_{\epsilon})$. Then for all $t > T$ we deduce that
	\begin{eqnarray*}
		\vert\vert  f(t+\omega,u( \big[ t+\omega\big])) - f(t,u( \big[ t \big])\vert\vert
		& \leq & \frac{\epsilon}{2}+ L\frac{\epsilon}{2L}\\
		& \leq & \epsilon.
	\end{eqnarray*}

	$(2)$ According to $(1)$ we have 
	$$\lim_{t \rightarrow \infty}(f(t+\omega,u([t+\omega])-f(t,u([t])))=0,$$
	meaning that $$\forall \epsilon^{1/p}>0, \, \exists T >0, \; t \geq T$$
	$$ \Rightarrow 
	||f(t+\omega,u([t+\omega]))- f(t,u([t]))|| \leq      \epsilon^{1/p}. $$
	The function $t\to f(t,u[t])$ is  continuous on every intervals $]n,+1[$ but $\displaystyle{\lim_{ t \rightarrow  n^-}}  f(t,u( \big[ t \big]))=f\big( n,u(n-1)\big)$ and $\displaystyle{\lim_{ t \rightarrow n^{+}}}f(t,u([t])=f(n,u(n))$. Therefore the function $t\to f(t,u[t])$ is a piecewise continuous function and it is measurable on	$\R_{+}$. Then for $t\geq [T]+1$, we have 
	\begin{eqnarray*}
		&\;&\int_{t}^{t+1} ||f(s,u([s+\omega]))-f(s,u([s]))||^{p}\\
	    & \leq & \int_{t}^{t+1} \epsilon\;\; ds \\
		& \leq & \epsilon.  
	\end{eqnarray*}
	
	$(3)$ Since the function  $t \rightarrow f(t,u([t]))$ is not continuous on $\R_{+}$, it can't be $S$-asymptotically $\omega$-periodic.
\end{proof}

\begin{lemma}
	\label{limit2f}
	Let $\omega \in \mathbb{N}^*$. Assume that $f: \mathbb{R}^{+}\times \mathbb{X} \rightarrow  \mathbb{X}$ is uniformly $S$-asymptotically $\omega$-periodic on bounded sets in the Stepanov sense and asymptotically uniformly continuous on bounded sets in the Stepanov sense. Let $u:\mathbb{R}^+ \rightarrow \mathbb{X}$ be a function in $SAP_{\omega}(\mathbb{R}^{+}, \mathbb{X})$, and let $v(t)=f(t,u([t]))$. Then $v\in S^{p}SAP_{\omega}(\mathbb{R}^{+}, \mathbb{X}).$	
\end{lemma}





\begin{proof}	
	Set $B =: \mathcal{R}(u) = \{ u[t], \; t\geq 0  \} \subset \mathbb{X}$. \\
	Since $f$ is uniformly $S$-asymptotically $\omega$-periodic on bounded sets in the Stepanov sense, there exist functions $g_{B} \in BS^{p}(\mathbb{R}^{+}, \mathbb{R}) $ and $h_{B} \in BS_{0}^{p} (\mathbb{R}^{+},\mathbb{R})$ satisfying the properties involved in Definition 2.6 and 2.8 in relation with the set $B =: \mathcal{R}(u)$. \\
	The function $v$ belongs to $ BS^{p}(\mathbb{R}^{+},\mathbb{X})$  because 
	\begin{eqnarray*}
		\int_{t}^{t+1} ||v(\tau )||^{p}d\tau &=& \int_{t}^{t+1} ||f(\tau, u([\tau]))||^{p}d\tau\\
		& \leq& \int_{t}^{t+1} ||g_{B}(\tau)||^{p}d\tau\\
		&\leq& \sup_{t\geq 0} \Big(  \int_{t}^{t+1} ||g_{B}(\tau)||^{p}d\tau   \Big).  
	\end{eqnarray*}    
	
	Therefore
	$$  ||v^{b}||_{\mathbb{L}^{\infty}(\mathbb{R}^{+},L^{p})} \leq  ||g_{B}||_{S^{p}}. $$
	We  have for all $t \geq 0$ :
	$$ \int_{t}^{t+1} || f(s+\omega, u([s+\omega]))- f(s, u([s+\omega]))||^{p} ds$$
    $$ \leq \int_{t}^{t+1}|| h_{B}(s)||^{p}ds.         $$ 
	Note that  $h_{B} \in BS_{0}^{p} (\mathbb{R}^{+},\mathbb{R})$; this implies that for $\epsilon > 0$
	there exists $t'_{\epsilon} >  0$ such that for all $t \geq t'_{\epsilon}$ we have 
	$$  \int_{t}^{t+1}|| h_{B}(s)||^{p}ds \leq \epsilon^{p}/2 .                      $$
	Thus $$ \int_{t}^{t+1} || f(s+\omega, u([s+\omega]))- f(s, u([s+\omega]))||^{p} ds$$ $$\leq \epsilon^{p}/2 \quad  \textrm{for all} \; t\geq t'_{\epsilon}(\ast). $$
	Furthermore since $f$ is asymptotically uniformly continuous on bounded sets in the Stepanov sense, thus
	for all $\epsilon > 0$, theres exists $t_{\epsilon} \geq 0 $ and $\delta_{\epsilon} > 0$ such that 
	$$ \int_{t}^{t+1} || f(s, u([s+\omega]))- f(s, u([s]))||^{p} ds$$
	$$ \leq \epsilon^{p}/2  \quad \textrm{for all } \; t \geq t_{\epsilon} \quad (\ast\ast) $$ 
	because  $$     ||u([s+\omega])-u([s]) || \leq \delta_{\epsilon}.        $$
	
	The estimates $(\ast)$  and $(\ast\ast) $ lead to 
	
	\begin{eqnarray*}
		&\,& \int_{t}^{t+1} || v(s+\omega)- v(s) ||^{p} ds\qquad\qquad\\
		 &= &\int_{t}^{t+1} || f(s+\omega, u([s+\omega]))- f(s, u([s]))||^{p} ds\\
		& \leq &  \int_{t}^{t+1} || f(s+\omega, u([s+\omega]))\\
		&-& f(s, u([s+\omega]))||^{p} ds\\
		&+& \int_{t}^{t+1} || f(s, u([s+\omega]))- f(s, u([s]))||^{p} ds\\
		& \leq & \epsilon^{p}/2  + \epsilon^{p}/2  = \epsilon^{p}.
	\end{eqnarray*}
	
	Therefore for all $\epsilon >0$ there exists $T_{\epsilon} = Max ( t_{\epsilon}, t'_{\epsilon}) >0$ such that 
	for all $t\geq T_{\epsilon}$ we have  $$ \Big (\int_{t}^{t+1} || v(s+\omega)- v(s) ||^{p} ds \Big)^{1/p} \leq \epsilon.$$
	
	We conclude that $v \in S^{p}SAP_{\omega}(\mathbb{R}^{+},\mathbb{X})$.

\end{proof}
\section{Main Results}

\begin{definition}
	A solution of (\ref{eqn: eq3050}) on $\mathbb{R}^+$ is a function $x(t)$ that satisfies the conditions:
	\begin{itemize}
		\item[(1)] $x(t)$ is continuous on $\mathbb{R}^+$. 
		\item[(2)] The derivative $x'(t)$ exists at each point $t \in \mathbb{R}^+$, with possible exception at the points $[t],\;\, t \in \mathbb{R}^+$ where one-sided derivatives exists.
		\item[(3)] The equation (\ref{eqn: eq3050}) is satisfied on each interval $[n,n + 1[$  with $n \in \mathbb{N}.$
	\end{itemize}
\end{definition}

Now we make the following hypothesis:\\

{\bf (H1)} : The function $f$ is uniformly $S$-asymptotically $\omega$-periodic on bounded sets in the Stepanov sense and satisfies the Lipschitz condition
$$\vert\vert f(t,u)-f(t,v)\vert\vert \le L \vert\vert u-v\vert\vert,\;\;u,\;v \in\mathbb{X},\;t \in \mathbb{R}^+.$$\\
We assume that $A(t)$ generates an  evolutionary process $(U(t,s))_{t\geq s}$ in $\mathbb{X}$, that is, a two-parameter family of bounded linear operators that satisfies the following conditions:
\begin{itemize}
	\item[1.] $U(t,t)=I$ for all $t\geq 0$ where $I$ is the identity operator.
	\item[2.] $U(t,s)U(s,r)=U(t,r)$ for all $t \geq s \geq r.$
	\item[3.] The map $(t,s)\mapsto U(t,s)x$ is continuous for every fixed $x\in \mathbb{X}$.
\end{itemize}

Then the function $g$ defined by $g(s) = U(t,s)x(s)$, where $x$ is a solution of  (\ref{eqn: eq3050}), is differentiable for $s < t$.
\begin{eqnarray*}
	\frac{dg(s)}{ds} &=& -A(s)U(t,s)x(s) + U(t,s)x'(s) \\
	&=& -A(s)U(t,s)x(s)+U(t,s)A(s)x(s)\\
	&+& U(t,s)f(s,x([s]))\\
	&= & U(t,s)f(s,x([s])).
\end{eqnarray*}
\begin{eqnarray}
\label{eq:integrale}
\frac{dg(s)}{ds} &=& U(t,s)f(s,x([s])).\qquad\qquad\qquad\qquad\qquad\qquad\;
\end{eqnarray}
The function $x([s])$ is a step function. By { \bf (H1)}, $f(s,x([s]))$
is piecewise continuous. Therefore  $f(s,x([s]))$ is integrable on $[0,t]$ where $t\in \mathbb{R}^+$. Integrating (\ref{eq:integrale}) on $[0,t]$ we obtain that

$$ x(t)- U(t,0)c_0 =  \int_{0}^{t} U(t,s)f(s,x([s]))ds. $$

Therefore, we define\\

\begin{definition}
We assume {\bf (H1)} is satisfied and that $A(t)$ generates an  evolutionary process $(U(t,s))_{t\geq s}$ in $\mathbb{X}$. The continuous function $x$ given by 
	$$ x(t) = U(t,0)c_{0} + \int_{0}^{t} U(t,s)f(s,x([s]))ds $$
	is called the mild solution of equation (\ref{eqn: eq3050}).
\end{definition}
Now we make the following hypothesis.\\

{ \bf (H2)}: $A(t)$ generates a $\omega$-periodic ($\omega >0$) exponentially stable evolutionnary process $(U(t,s))_{t\geq s}$ in $\mathbb{X}$, that is, a two-parameter family of bounded linear operators that 
satisfies the following conditions:
\begin{itemize}
   \item[1.]  For all $t\geq 0$,
   $$ U(t,t)=I \; \textrm{where} \; I\; \textrm{ is the identity operator}.$$
   \item[2.] For all $t \geq s \geq r,$
 $$U(t,s)U(s,r)=U(t,r).$$ 
   \item[3.] The map $(t,s)\mapsto U(t,s)x$ is continuous for every fixed $x\in \mathbb{X}$.
   \item[4.] For all  $t\geq s$,  $$ U(t+\omega, s+\omega)= U(t,s)$$ 
  ($\omega$-periodicity).
   \item[5.] There exist $K>0$ and $a>0$ such that
   $$||U(t,s)|| \leq Ke^{-a(t-s)}$$ for $t\geq s.$
\end{itemize}

\begin{theorem}
	\label{lem1}
We assume that { \bf (H2)} is satisfied and that $f\in S^{p}SAP_{\omega}(\mathbb{R}^{+}, \mathbb{X}).$ Then
 $$(\wedge f)(t)=\int_{0}^t U(t,s)f(s)ds \in SAP_{\omega}(\mathbb{R}^{+}, \mathbb{X}), t \in \mathbb{R}^+.$$
\end{theorem}

\begin{proof}
	Let $ u(t)= \int_{0}^t U(t,s)f(s)ds.$\\	
For $n \le t \le n+1, \; n \in \mathbb{N}$, we observe  that 
	\begin{eqnarray*}
		&\;& || u(t) ||\\
		 &  \le &	 \int_{0}^t || U(t,s)f(s) ||\;ds \\
		& \le & \int_{0}^{n} || U(t,s)f(s) ||\;ds + \int_{n}^{t} || U(t,s)f(s) ||\;ds\\
		& \le & \int_{0}^{n} Me^{-a(t-s)}|| f(s) ||\;ds\\ 
		&+ &  \int_{n}^{t} Me^{-a(t-s)}|| f(s) ||\;ds\\
		& \le & \int_{0}^{n} Me^{-a(n-s)}|| f(s) ||\;ds +  \int_{n}^{t} M|| f(s) ||\;ds\\
		& \le & \sum_{k=0}^{n-1}\int_{k}^{k+1} Me^{-a(n-s)}|| f(s) ||\;ds  +  \int_{n}^{t} M|| f(s) ||\;ds \\
		& \le & \sum_{k=0}^{n-1}\int_{k}^{k+1} Me^{-a(n-k-1)}|| f(s) ||\;ds\\  
		&+&  \int_{n}^{n+1} M|| f(s) ||\;ds \\
		& \le & \sum_{k=0}^{n-1} Me^{-a(n-k-1)} \int_{k}^{k+1}|| f(s) ||\;ds\\  
		&+&  M \int_{n}^{n+1} || f(s) ||\;ds \\
		& \le & \sum_{k=0}^{n-1} Me^{-a(n-k-1)}\Big( \int_{k}^{k+1}|| f(s) ||^p\;ds\Big)^{\frac{1}{p}}\\
		& +& M \Big(\int_{n}^{n+1} || f(s) ||^p\;ds \Big)^{\frac{1}{p}} \\
		& \le & M\Big( \sum_{j=0}^{\infty} e^{-aj} +1 \Big) \vert\vert f \vert\vert_{S^p}\\
		& \le & M\Big( \frac{2-e^{-a}}{1-e^{-a}} \Big)  \vert\vert f \vert\vert_{S^p}.
	\end{eqnarray*}
	Therefore $u$ is bounded.
	$$\textrm{Now, show that} \;\; \displaystyle{\lim_{t \rightarrow \infty}} u(t+\omega)-u(t)=0.$$ We have	 	     
	\begin{eqnarray*}
		u(t+\omega)-u(t) &=& \int_{0}^{t+\omega} U(t+\omega,s)f(s)ds\\ 
		&-& \int_{0}^{t} U(t,s)f(s)ds\\
		&=& \int_{0}^{\omega} U(t+\omega,s)f(s)ds\\ 
		&+& \int_{\omega}^{t+\omega} U(t+\omega,s)f(s)ds \\
		&-& \int_{0}^{t} U(t,s)f(s)ds\\
		&=& I_{1}(t)+ I_{2}(t),
	\end{eqnarray*}
	where $$  I_{1}(t)=  \int_{0}^{\omega} U(t+\omega,s)f(s)ds,      $$
	and $$ I_{2}(t)= \int_{\omega}^{t+\omega} U(t+\omega,s)f(s)ds 
	-\int_{0}^{t} U(t,s)f(s)ds.     $$
	We note that 
	$$I_{1}(t)= U(t+\omega, \omega)  \int_{0}^{\omega} U(\omega,s)f(s)ds  = U(t+\omega,\omega)u(\omega),             $$
	and by using the fact that $(U(t,s))_{t\geq s}$  is exponentially stable, we obtain
	$$||I_{1}(t)|| \leq Ke^{-at}||u(\omega)||,        $$
	which shows that 
	$$\displaystyle{\lim_{t\to \infty}} I_{1}(t)=0.$$
	
	Let $\epsilon >0$. Since $f\in S^{p}SAP_{\omega}(\mathbb{R}^{+}, \mathbb{X})$, there exists $m\in\mathbb{N}$ such that for $t \ge m$
	
	$$\Big(\int_t^{t+1}\vert\vert f(s+\omega)-f(s) \vert\vert^p ds\Big)^{\frac{1}{p}} < \epsilon .$$
	
	For $m\le n \le t \le n+1$, we have
	
	$$	I_2(t) =  \int_0^t U(t,s)\big( f(s+\omega)-f(s)\big) ds \,  $$
		$$ \le  I_{2,1}(t)+I_{2,2}(t)+I_{2,3}(t),$$
	where 
	\begin{equation*}
		\left\{
		\begin{array}{l}
			I_{2,1}(t)  =  \int_0^{m} 
			U(t,s) \Big( f(s+\omega)-f(s) \Big)\,ds \\
			I_{2,2}(t) =   \sum_{k= m}^{n-1} \int_k^{k+1} U(t,s) \Big( f(s+\omega)-f(s) \Big)\,ds,\\
			I_{2,3}(t)  =  \int_n^t U(t,s) \Big( f(s+\omega)-f(s) \Big)\,ds.
		\end{array}
		\right.
	\end{equation*}
	We observe that
	\begin{eqnarray*}
		\vert\vert I_{2,1}(t) \vert\vert & \le &  \int_0^{m} \vert\vert
		U(t,s) \vert\vert\,\vert\vert f(s+\omega)-f(s) \vert\vert\,ds \\	
		& \le & Me^{-a(t-m)} \int_0^m \vert\vert f(s+\omega)- f(s)\vert\vert ds.
	\end{eqnarray*}
	Therefore, there exists $\nu_m \in \mathbb{N}$, $\nu_m \ge m$ such that for $t\ge \nu_m$
	$$\vert\vert I_{2,1}(t) \vert\vert \le \epsilon.$$
	
	
	Using Holder's inequality, we observe also that
	\begin{eqnarray*}
		&\;& || I_{2,2}(t) ||\\
	 & \le & \sum_{k=m}^{n-1} \int_k^{k+1} 
		|| U(t,s)||\, || f(s+\omega)-f(s) ||\,ds\\
		& \le & \sum_{k=m}^{n-1} M \int_k^{k+1} 
		e^{-a(t-s)}\, || f(s+\omega)-f(s) ||\,ds\\
		& \le & \sum_{k=m}^{n-1} M \int_k^{k+1} 
		e^{-a(n-k-1)}\, || f(s+\omega)-f(s) ||\,ds\\
		& \le & M  \sum_{k=m}^{n-1} e^{-a(n-k-1)} \int_k^{k+1} || f(s+\omega)-f(s) ||\,ds\\
		& \le & M  \sum_{k=m}^{n-1} e^{-a(n-k-1)} \Big( \int_k^{k+1} || f(s+\omega)-f(s) ||^p\,ds \Big)^{\frac{1}{p}}\\
		& \le & M \big( e^{-a(n-m-1)}+e^{-a(n-m-2)}+...+1\big)\epsilon\\
		& \le & \frac{M}{1-e^{-a}}\epsilon. 
	\end{eqnarray*}	
	We observe also that
	\begin{eqnarray*}
		|| I_{2,3}(t) || & \le & \int_n^t || U(t,s)||\, || f(s+\omega)-f(s) ||\,ds\\	
		& \le & \int_n^t  Me^{-a(t-s)}\, || f(s+\omega)-f(s) ||\,ds\\	
		& \le & M \int_n^t || f(s+\omega)-f(s) ||\,ds\\
		& \le & M \int_n^{n+1} || f(s+\omega)-f(s) ||\,ds\\	
		& \le & M \Big(\int_n^{n+1} || f(s+\omega)-f(s) ||^p\,ds\Big)^{\frac{1}{p}}\\
		& \le & M \epsilon .
	\end{eqnarray*}	
	Finally, for $t \ge \nu_m$
	\begin{eqnarray*}
		|| I_{2}(t) || & \le & || I_{2,1}(t) || + || I_{2,2}(t) || + || I_{2,3}(t) ||\\
		& \le &\Big( 1 + \frac{M}{1-e^{-a}} + M   \Big)\epsilon,
	\end{eqnarray*}
	thus $\displaystyle{\lim_{t \rightarrow \infty}} I_2(t)=0$.	
	We conclude that $u\in SAP_{\omega}(\mathbb{R}^{+}, \mathbb{X})$.
\end{proof}

Now we make the following hypothesis.\\


\begin{theorem}	\label{lem2}
	Let $\omega \in \mathbb{N}^*$. We assume that the hypothesis {\bf (H1)} and {\bf (H2)} are satisfied.  Then (\ref{eqn: eq3050}) has a unique $S$-asymptotically $\omega$-periodic mild solution provided that
	$$\Theta:=    \frac{LM}{a}   <1.$$ 
\end{theorem}
\begin{proof}
	We define the nonlinear operator $\Gamma$ by the expression
	\begin{eqnarray*}
		(\Gamma \phi)(t)& = &U(t,0)c_{0} + \int_{0}^{t} U(t,s)f(s,\phi([s]))ds \\
		&= & U(t,0)c_0 + (\wedge_1 \phi)(t), 
	\end{eqnarray*}
	where
	$$(\wedge_1 \phi)(t)=\int_{0}^{t} U(t,s)f(s,\phi([s])).$$
	
	According to the hypothesis {\bf(H2)}, we have
	$$ ||U(t+\omega,0)-U(t,0)||  \leq ||U(t+\omega,0)|| + ||U(t,0)||$$
	$$ \qquad\qquad\qquad\qquad \leq Ke^{-a(t+\omega)} + Ke^{-at}. $$
	Therefore $\displaystyle{\lim_{t\to \infty}} ||U(t+\omega,0)-U(t,0)|| = 0$.\\
	According to the Lemma \ref{limit1f} (resp. lemma \ref{limit2f}) the function $t \rightarrow f(t,\phi( \big[ t\big]))$ belongs to  $S^{p}SAP_{\omega}(\mathbb{R}^{+}, \mathbb{X}).$ According to the Theorem \ref{lem1} the operator $\wedge_1$ maps $SAP_{\omega}(\mathbb{R}^{+}, \mathbb{X})$ into itself. Therefore the operator $\Gamma $ maps $SAP_{\omega}(\mathbb{R}^{+}, \mathbb{X})$ into itself.\\	
	We have 
	\begin{eqnarray*}
		&\;&||(\Gamma \phi)(t)- \Gamma \psi)(t)||\\
	    &= &  \left|\left|   \int_{0}^{t} U(t,s)\big( f(s,\phi([s]))-f(s,\psi([s])) \big)ds \right| \right| \\
		& \leq &  
		\int_{0}^{t} || U(t,s)||\, ||f(s,\phi([s]))-f(s,\psi([s]))|| ds \\
		& \leq &   L \int_{0}^{t} || U(t,s)||\,  ||\phi([s])-\psi([s])|| ds \\
		& \leq & LM \int_{0}^{t}e^{-a(t-s)} \,  ||\phi([s])-\psi([s])|| ds  \\
		& \leq & LM \int_{0}^{t}e^{-a(t-s)} \, ||\phi-\psi||_{\infty} ds \\
		& \leq & LM\frac{1-e^{-at}}{a}||\phi-\psi||_{\infty}\\
		& \leq & \frac{LM}{a} ||\phi-\psi||_{\infty}.
	\end{eqnarray*}
	Hence we have : 
	$$ ||\Gamma \phi- \Gamma \psi ||_{\infty} \leq   \frac{LM}{a} ||\phi-\psi||_{\infty}.  $$
	This proves that $\Gamma$ is a contraction and we conclude that $\Gamma$ has a unique fixed point in 
	$SAP_{\omega}(\mathbb{R}^{+}, \mathbb{X})$. The proof is complete.	
	
\end{proof}

\begin{center}
	\section{Application}		
\end{center}
Consider the following heat equation with Dirichlet conditions:
\begin{equation}
\label{eqn: eq3055}
\left\{
\begin{array}{l}
\frac{\partial u(t,x)}{\partial t}= \frac{\partial^2 u(t,x)}{\partial x^2}+(-3+\sin(\pi t))u(t,x)+f(t,u([t],x)), \\
u(t,0)=u(t,\pi)=0, t\in \mathbb{R}^+,\\
u(0,x)=c_0,
\end{array}
\right.
\end{equation}
where $c_0 \in L^2[0, \pi]$ and the function $f$ is uniformly $S$-asymptotically $\omega$-periodic on bounded sets and satisfies the lipschitz condition, that is, there exists a constant $L>0$ such that
$$ \vert\vert f(t,x)-f(t,y) \vert\vert\le L \vert\vert x-y \vert\vert,\;\;\forall t \ge 0,\,\forall x,y \in \mathbb{X}.$$
Let $\mathbb{X}=L^2[0,\pi]$ be endowed with it's natural topology. Define
$$\mathcal{D}(A)=\{u \in L^2[0, \pi] \;such\, that\; u'' \in L^2[0, \pi]$$
$$\, and\, u(0)=u(\pi)=0\}$$
$$A u =u''\,for\,all\, u \in \mathcal{D}(A).$$
Let $\phi_n(t)=\sqrt{\frac{2}{\pi}}\sin(nt)$ for all $n \in \mathbb{N}$. $\phi_n$ are eigenfunctions of the operator $(A, \mathcal{D}(A))$ with eigenvalues $\lambda_n=-n^2$. $A$ is the infinitesimal generator of a semi-group $T(t)$ of the form
$$T(t)\phi = \sum_{n=1}^{\infty} e^{-n^2 t}\langle \phi, \phi_n \rangle \phi_n, \;\forall \phi \in L^2[0, \pi]$$
and 
$$\vert\vert T(t) \vert\vert \le e^{-t}, \; for\,t\ge0 $$
(see \cite{rong},\cite{Xia}).\\
Now define $A(t)$ by:
\begin{equation*}
	\left\{
	\begin{array}{l}
		\mathcal{D}(A(t))=\mathcal{D}(A) \\
		A(t)=A+q(t,x),
	\end{array}
	\right.
\end{equation*}
where $q(t,x)=-3+\sin(\pi t)$.\\
Note that $A(t)$ generates an evolutionnary process $U(t,s)$ of the form
$$U(t,s)=T(t-s)e^{\int_s^t q(,v,x)dx}.$$
Since $q(t,x)=-3+\sin( \pi t)\le -2$, we have
$$U(t,s) \le T(t-s)e^{-2(t-s)}$$
and 
$$\vert\vert U(t,s)\vert\vert \le \vert\vert T(t-s) \vert\vert e^{-é(t-s)}\le e^{-3(t-s)}.$$
Since $q(t+2,x)=q(t,x)$, we conclude that $U(t,s)$ is a $2$-periodic evolutionnary process exponentially stable.

The equation (\ref{eqn: eq3055}) is of the form
\begin{equation*}
	\left\{
	\begin{array}{l}
		x'(t)=A(t)x(t)+f(t,x([t]) ), \\
		x(0)= c_{0}.
	\end{array}
	\right.
\end{equation*}
By Theorem \ref{lem2}, we claim that
\begin{theorem}
	If $L<3$ then the equation	(\ref{eqn: eq3055}) admits an unique mild solution $u(t)\in SAP_{\omega}(\R^+, \mathbb{X})$.
\end{theorem}

\end{document}